\documentclass[11pt,leqno]{amsart}

\usepackage{amsmath,amssymb,amsthm,mathrsfs}
\usepackage{hyperref}
\usepackage{tikz}
\usetikzlibrary{arrows}
\usetikzlibrary{decorations.markings}
\tikzset{nimplies/.style={decoration={markings,mark= at position 0.5 with {\node[transform shape] (tempnode){\tiny/};}},postaction={decorate}}}
\tikzset{negated/.style={decoration={markings,mark= at position 0.5 with {\node[transform shape] (tempnode) {$\backslash$};}},postaction={decorate}}}

\newtheorem{theorem}{Theorem}[section]
\newtheorem*{theorem*}{Theorem}

\newtheorem{question}[theorem]{Question}

\newtheorem{proposition}[theorem]{Proposition}
\newtheorem{corollary}[theorem]{Corollary}
\newtheorem{fact}[theorem]{Fact}
\theoremstyle{definition}
\newtheorem{definition}[theorem]{Definition}
\newtheorem{example}[theorem]{Example}
\theoremstyle{remark}
\newtheorem{remark}[theorem]{Remark}
\numberwithin{equation}{section}

\DeclareMathOperator{\ssdtwop}{SSD2P}
\newcommand{\SSDTWOP}[1]{\ssdtwop_{#1}}
\DeclareMathOperator{\assdtwop}{1-ASSD2P}
\newcommand{\ASSDTWOP}[1]{\assdtwop_{#1}}
\DeclareMathOperator{\sdtwop}{SD2P}
\newcommand{\SDTWOP}[1]{\sdtwop_{#1}}
\DeclareMathOperator{\asdtwop}{1-ASD2P}
\newcommand{\ASDTWOP}[1]{\asdtwop_{#1}}
\DeclareMathOperator{\asq}{ASQ}
\newcommand{\ASQ}[1]{\asq_{#1}}
\DeclareMathOperator{\sq}{SQ}
\newcommand{\SQ}[1]{\sq_{#1}}
\DeclareMathOperator{\density}{d}
\DeclareMathOperator{\cellularity}{c}
\DeclareMathOperator{\cofinality}{cf}
\begin{document}
\title{On the transfinite symmetric strong diameter two property}

\author{\href{https://orcid.org/0000-0002-1500-8123}{Stefano Ciaci}}
\address{Institute of Mathematics and Statistics, University of Tartu, Narva mnt 18, 51009 Tartu, Estonia}
\email{stefano.ciaci@ut.ee}
\urladdr{\url{https://stefanociaci.science.blog/}}

\thanks{This work was supported by the Estonian Research Council grants (PSG487) and (PRG1501).}
\begin{abstract} 
    We study transfinite analogues of the \emph{symmetric strong diameter two property}. We investigate the stability of these properties under $c_0$, $\ell_\infty$ sums and under projective tensor products. Moreover, we characterize Banach spaces of the form $C_0(X)$, where $X$ is a T$_4$ locally compact space, which posses these transfinite properties via cardinal functions over $X$. As an application, we are able to produce a variety of examples of Banach spaces which enjoy or fail these properties.
\end{abstract}

\maketitle
\section{Introduction}\label{section: introduction}

Given an infinite-dimensional real Banach space $X$, its topological dual, its unit ball and its unit sphere are denoted by $X^*$, $B_X$ and $S_X$, respectively.

\begin{definition}\label{definition: SSD2P}
    A Banach space $X$ has the \emph{symmetric strong diameter two property} ($\ssdtwop$) if, and only if, for every $x^*_1,\ldots,x^*_n\in S_{X^*}$ and $\varepsilon>0$, there are $x_1,\ldots,x_n,y\in B_X$ such that $\|y\|\ge1-\varepsilon$, $x_i\pm y\in B_X$ and $x^*_i(x_i)\ge1-\varepsilon$ for all $1\le i\le n$.
\end{definition}

The $\ssdtwop$ was introduced in \cite{ANP2019_NewApplicationsOfExtremelyRegularFunctionSpaces}, but the original definition contains the additional requirement that $x^*_i(x_i\pm y)\ge1-\varepsilon$, which is redundant. Indeed, if we require that $x^*_i(x_i)\ge1-\varepsilon/2$, since $x_i\pm y\in B_X$, then $|x^*_i(y)|\le\varepsilon/2$ and, therefore, Definition~\ref{definition: SSD2P} is equivalent to the original one.

Examples of Banach spaces enjoying the $\ssdtwop$ include Lindenstrauss spaces, uniform algebras, almost square Banach spaces, Banach spaces with an infinite dimensional centralizer, somewhat regular subspaces of $C_0(X)$ spaces, where $X$ is an infinite locally compact Hausdorff space, and Müntz spaces (see \cite{HLLN2019_SymmetricStrongDiameterTwoProperty}).

In \cite{ACLLR2022_TransfiniteAlmostSquareBanachSpace} transfinite analogues of the $\ssdtwop$ were defined, but, before recalling these definitions, let us introduce some notation. Given $r\in\mathbb (0,1)$, $B\subset B_X$ and $A\subset S_{X^*}$, we say that $B$ \emph{$r$-norms} $A$ if, for every $x^*\in A$, there is $x\in B$ such that $x^*(x)\ge r$. In addition, we say that $B$ \emph{norms} $A$ if it $r$-norms it for all $r\in (0,1)$.

\begin{definition}\label{definition: k-SSD2P}\cite[Definition~5.3]{ACLLR2022_TransfiniteAlmostSquareBanachSpace}
    Let $X$ be a Banach space and $\kappa$ an infinite cardinal.
    \begin{itemize}
        \item[(i)] $X$ has the \emph{$\SSDTWOP{\kappa}$} if, for every set $A\subset S_{X^*}$ of cardinality $<\kappa$ and $\varepsilon>0$, there are $B\subset B_X$, which $(1-\varepsilon)$-norms $A$, and $y\in B_X$ satisfying $B\pm y\subset B_X$ with $\|y\|\ge1-\varepsilon$.
        \item[(ii)] $X$ has the \emph{$\ASSDTWOP{\kappa}$} if, for every set $A\subset S_{X^*}$ of cardinality $<\kappa$, there are $B\subset S_X$, which norms $A$, and $y\in S_X$ satisfying $B\pm y\subset S_X$. 
    \end{itemize}   
\end{definition}

In the following, we aim to investigate these transfinite extensions of the $\ssdtwop$ and, in particular, to show differences in their behaviour when compared to the regular $\ssdtwop$.

Now, let us also recall the transfinite extensions of almost squareness and the strong diameter two property.

\begin{definition}\label{definition: k-ASQ}\cite[Definition~2.1]{ACLLR2022_TransfiniteAlmostSquareBanachSpace}
    Let $X$ be a Banach space and $\kappa$ a cardinal.
    \begin{itemize}
        \item[(i)] $X$ is \emph{$\ASQ{\kappa}$} if, for every set $A\subset S_X$ of cardinality $<\kappa$ and $\varepsilon>0$, there exists $y\in S_X$ such that $\|x\pm y\|\le1+\varepsilon$ holds for all $x\in A$.
        \item[(ii)] $X$ is \emph{$\SQ{\kappa}$} if, for every set $A\subset S_X$ of cardinality $<\kappa$, there exists $y\in S_X$ such that $\|x\pm y\|\le1$ holds for all $x\in A$.
    \end{itemize}
\end{definition}

\begin{definition}\label{definition: k-SD2P}\cite[Definitions~2.11 and 2.12]{CLL2022_AttainingStrongDiameterTwoPropertyForInfiniteCardinals}
    Let $X$ be a Banach space and $\kappa$ an infinite cardinal.
    \begin{itemize}
        \item[(i)] $X$ has the $\SDTWOP{\kappa}$ if, for every set $A\subset S_{X^*}$ of cardinality $<\kappa$ and $\varepsilon>0$, there are $B\subset B_X$, which $(1-\varepsilon)$-norms $A$, and $x^*\in B_{X^*}$ satisfying $x^*(x)\ge1-\varepsilon$ for all $x\in B$.
        \item[(ii)] $X$ has the $\ASDTWOP{\kappa}$ if, for every set $A\subset S_{X^*}$ of cardinality $<\kappa$, there are $B\subset S_X$, which norms $A$, and $x^*\in S_{X^*}$ satisfying $x^*(x)=1$ for all $x\in B$.
    \end{itemize}   
\end{definition}

It is clear that every $\ASQ{\kappa}$ ($\SQ{\kappa}$, respectively) Banach space enjoys the $\SSDTWOP{\kappa}$ ($\ASSDTWOP{\kappa}$, respectively). Moreover, it was shown in  \cite[Proposition~5.4]{ACLLR2022_TransfiniteAlmostSquareBanachSpace} that the $\SSDTWOP{\kappa}$ ($\ASSDTWOP{\kappa}$, respectively) implies the $\SDTWOP{\kappa}$ ($\ASDTWOP{\kappa}$, respectively). To sum up, the following implications hold true.

\begin{figure}[h]
\centering
\begin{tikzpicture}
\path
(0,0) node (1) {$\ASQ{\kappa}$}
(0,1.2) node (2) {$\SQ{\kappa}$}
(2.8,0) node (3) {$\SSDTWOP{\kappa}$}
(2.8,1.2) node (4) {$\ASSDTWOP{\kappa}$}
(5.6,0) node (5) {$\SDTWOP{\kappa}$}
(5.6,1.2) node (6) {$\ASDTWOP{\kappa}$};
\draw[-implies,double equal sign distance](2.south)--(1.north);
\draw[-implies,double equal sign distance](4.south)--(3.north);
\draw[-implies,double equal sign distance](6.south)--(5.north);
\draw[-implies,double equal sign distance](2.east)--(4.west);
\draw[-implies,double equal sign distance](1.east)--(3.west);
\draw[-implies,double equal sign distance](4.east)--(6.west);
\draw[-implies,double equal sign distance](3.east)--(5.west);
\end{tikzpicture}
\end{figure}

\subsection{Content of the paper}\label{subsection: content of the paper}\hfill\\

In Section~\ref{section: stability results}, we study the stability of the transfinite $\ssdtwop$ with respect to operations between Banach spaces.

We provide a complete description concerning $c_0$ and $\ell_\infty$ sums (see Theorems~\ref{theorem: c_0 sums} and \ref{theorem: ell_infty sums}), which informally state that these sums of Banach spaces enjoy the $\SSDTWOP{\kappa}$ if, and only if, we can always find one component which satisfies a property which arbitrarily well approximates the $\SSDTWOP{\kappa}$. Thanks to these characterizations, we show that, for example, the Banach spaces $c_0(\mathbb N,\ell_n(\kappa))$ and $\ell_\infty(\mathbb N,\ell_n(\kappa))$ enjoy the $\SSDTWOP{\kappa}$ (see Example~\ref{example: c_0(ell_n) SSD2P}).

We also investigate the behaviour of the $\SSDTWOP{\kappa}$ under projective tensor products. Namely, we prove that the Banach space $X\hat{\otimes}_\pi Y$ has the $\SSDTWOP{\kappa}$, whenever $X$ and $Y$ enjoy the property.

We conclude this section by studying the difference in behaviour of the transfinite $\ssdtwop$ compared to the finite $\ssdtwop$. In particular, we prove that, for the transfinite case, it is not possible to replace the functionals with relatively weakly open sets in Definition~\ref{definition: k-SSD2P}, even though it is possible for the traditional $\ssdtwop$ (see Fact~\ref{fact: equivalent characterization SSD2P}~(ii)). Moreover, we prove that an equivalent internal description of the $\ssdtwop$ (see Fact~\ref{fact: equivalent characterization SSD2P}~(iii)) also fails in the transfinite case.

Section~\ref{section: C_0(X) spaces} is dedicated to extend the class of known examples which possess the transfinite $\ssdtwop$. To this aim, we search for a description in the class of $C_0(X)$ spaces, whenever $X$ is a T$_4$ locally compact space. The main result of this section states that the Banach space $C_0(X)$ fails the $\SSDTWOP{\kappa}$, where $\kappa$ is the successor cardinal of the density character of $X$, but it enjoys the $\ASSDTWOP{\mu}$, where $\mu$ is the cellularity of $X$ (see Theorem~\ref{theorem: C_0(X) SSD2P}).

Thanks to this result, new examples are provided, e.g. $C[0,1]$ and $\ell_\infty$ fail the $\SSDTWOP{\aleph_1}$, $C(\beta\mathbb N\setminus\mathbb N)$ enjoys the $\ASSDTWOP{2^{\aleph_0}}$ and $\ell_\infty(\kappa)$ has the $\ASSDTWOP{\kappa}$, whenever $\kappa>\aleph_0$.
\subsection{Notation}\label{subsection: notation}\hfill\\

Given a sequence of Banach spaces $(X_n)$ we define
\begin{equation*}
    \ell_\infty(\mathbb N,X_n):=\left\{x\in\prod_{n=1}^\infty X_n:(\forall n\in\mathbb N)\hspace{0.1cm}x(n)\in X_n\text{ and }\sup_{n}\|x(n)\|<\infty\right\}
\end{equation*}
endowed with the usual supremum norm. Moreover, we set
\begin{equation*}
    c_0(\mathbb N,X_n):=\left\{x\in \ell_\infty(\mathbb N,X_n):\lim_n\|x(n)\|=0\right\}.
\end{equation*}

Eventually, given a cardinal $\kappa$, we call $\cofinality(\kappa)$ its cofinality and $\kappa^+$ its successor cardinal.
\section{Stability results}\label{section: stability results}

In this section we investigate the behaviour of the transfinite $\ssdtwop$ with respect to operations between Banach spaces.
\subsection{Direct sums}\label{subsection: direct sums}

Given a sequence of Banach spaces $(X_n)$, it is known that the $c_0$ sum $c_0(\mathbb N,X_n)$ is always $\asq$ \cite[Example~3.1]{ALL2016_AlmostSquareBanachSpaces} and, therefore, has the $\ssdtwop$. Moreover, it was proved in \cite[Proposition~4.3]{ACLLR2022_TransfiniteAlmostSquareBanachSpace} that the $c_0$ sum $c_0(\mathscr A,X_\alpha)$ of a family of Banach spaces $\{X_\alpha:\alpha\in\mathscr A\}$ is $\SQ{|\mathscr A|}$ and thus has the $\ASSDTWOP{|\mathscr A|}$, whenever $|\mathscr A|>\aleph_0$. For these reasons, in the following we will focus only on countable $c_0$ sums with respect to the $\SSDTWOP{\kappa}$, for $\kappa>\aleph_0$.

\begin{theorem}\label{theorem: c_0 sums}
    Let $(X_n)$ be a sequence of Banach spaces and $\kappa>\aleph_0$. If, for every $r\in[0,1)$, there is $n\in\mathbb N$ such that, for every set $A\subset S_{X^*_n}$ of cardinality $<\kappa$, there exist $B\subset B_{X_n}$ and $y\in B_{X_n}$ such that $\|y\|\ge r$, $B$ $r$-norms $A$ and $B\pm y\subset B_{X_n}$, then $c_0(\mathbb N,X_n)$ enjoys the $\SSDTWOP{\kappa}$. If in addition $\cofinality(\kappa)>\aleph_0$, then the vice-versa also holds.
\end{theorem}
\begin{proof}
    Fix a set $A\subset S_{\ell_1(\mathbb N,X^*_n)}$ of cardinality $<\kappa$ and $\varepsilon>0$. Let $\{x^*_\alpha:\alpha\in\mathscr A\}$ be an enumeration of $A$ and find $m\in\mathbb N$ as in the statement for $r=1-\varepsilon$.

    By assumption there are $y^m\in B_{X_m}$ and $x^m_\alpha\in B_{X_m}$ such that $\|y^m\|\ge1-\varepsilon$, $x^m_\alpha\pm y^m\in B_{X_m}$ and $x^*_\alpha(m)(x^m_\alpha)\ge(1-\varepsilon)^\frac{1}{2}\|x^*_\alpha(m)\|$ hold for every $\alpha\in\mathscr A$.

    For each $m\not=n\in\mathbb N$ and $\alpha\in\mathscr A$, find $x^n_\alpha\in B_{X_n}$ satisfying $x^*_\alpha(n)(x^n_\alpha)\ge(1-\varepsilon)^\frac{1}{2}\|x^*_\alpha(n)\|$. Moreover, since $x^*_\alpha\in\ell_1(\mathbb N,X_n^*)$, there exists $n_\alpha\ge m$ such that 
    \begin{equation*}
        \sum_{1\le n\le n_\alpha}\|x^*_\alpha(n)\|\ge(1-\varepsilon)^\frac{1}{2}.
    \end{equation*}
    
    Now call
    \begin{equation*}
        x_\alpha:=\sum_{1\le n\le n_\alpha}x^n_\alpha e_n\in B_{c_0(\mathbb N,X_n)}
    \end{equation*}
    and $y:=y^me_m\in B_{c_0(\mathbb N,X_n)}$.

    Notice that
    \begin{equation*}
        x^*_\alpha(x_\alpha)=\sum_{1\le n\le n_\alpha}x^*_\alpha(n)(x^n_\alpha)\ge(1-\varepsilon)^\frac{1}{2}\sum_{1\le n\le n_\alpha}\|x^*_\alpha(n)\|\ge1-\varepsilon,
    \end{equation*}
    which means that the set $\{x_\alpha:\alpha\in\mathscr A\}$ $(1-\varepsilon)$-norms $A$.
    
    On the other hand, $\|y\|=\|y^m\|\ge1-\varepsilon$ and
    \begin{equation*}
        \|x_\alpha\pm y\|=\max\left\{\|x_\alpha^m\pm y^m\|,\sup_{1\le n\not=m\le n_\alpha}\|x_\alpha^n\|\right\}\le1
    \end{equation*}
    holds for all $\alpha\in\mathscr A$. Therefore, $c_0(\mathbb N,X_n)$ enjoys the $\SSDTWOP{\kappa}$.

    For the vice-versa, fix $\varepsilon>0$ and, for every $n\in\mathbb N$, $A_n\subset S_{X_n^*}$ of cardinality $<\kappa$. Call
    \begin{equation*}
        A:=\{x^*e_n:n\in\mathbb N\text{ and }x^*\in A_n\}\subset S_{\ell_1(\mathbb N,X_n^*)}
    \end{equation*}
    and notice that $|A|\le\aleph_0\cdot\sup|A_n|<\kappa$, because $\cofinality(\kappa)>\aleph_0$. Therefore, there exist a set $B\subset B_{c_0(\mathbb N,X_n)}$, which $(1-\varepsilon)$-norms $A$, and $y\in B_{c_0(\mathbb N,X_n)}$ such that $\|y\|\ge1-\varepsilon$ and $B\pm y\subset B_{c_0(\mathbb N,X_n)}$.

    Since $\|y\|\ge1-\varepsilon$, there exists $n\in\mathbb N$ satisfying $\|y(n)\|\ge1-\varepsilon$. Moreover, from the fact that, in particular, $B$ $(1-\varepsilon)$-norms the set $\{x^*\delta_n:x^*\in A_n\}$, we deduce that the set $B_n:=\{x(n):x\in B\}\subset B_{X_n}$ $(1-\varepsilon)$-norms $A_n$.
    
    Eventually, notice that, given $x(n)\in B_n$,
    \begin{equation*}
        \|x(n)\pm y(n)\|\le\|x\pm y\|\le1,
    \end{equation*}
    which concludes the proof.
\end{proof}

Notice that the same proof can be adjusted to $\ell_\infty$ sums too. As a matter of fact, it is not needed to find $n_\alpha$ as in the proof of Theorem~\ref{theorem: c_0 sums} and one can define
\begin{equation*}
    x_\alpha:=\sum_{n=1}^\infty x_\alpha^ne_n\in B_{\ell_\infty(\mathbb N,X_n)}.
\end{equation*}

By doing so, the following theorem is easily proved, up to a few minor changes.

\begin{theorem}\label{theorem: ell_infty sums}
    Let $\{X_\alpha:\alpha\in\mathscr A\}$ be a family of Banach spaces and $\kappa>\aleph_0$. If, for every $r\in[0,1)$, there is $\alpha\in\mathscr A$ such that, for every set $A\subset S_{X^*_\alpha}$ of cardinality $<\kappa$, there exist $B\subset B_{X_\alpha}$ and $y\in B_{X_\alpha}$ such that $\|y\|\ge r$, $B$ $r$-norms $A$ and $B\pm y\subset B_{X_\alpha}$, then $\ell_\infty(\mathscr A,X_\alpha)$ enjoys the $\SSDTWOP{\kappa}$. If in addition $\cofinality(\kappa)>|\mathscr A|$, then the vice-versa also holds.
\end{theorem}

\begin{corollary}\label{corollary: direct sum}
    Let $X$ and $Y$ be Banach spaces and $\kappa>\aleph_0$. Either $X$ or $Y$ enjoy the $\SSDTWOP{\kappa}$ if, and only if, $X\oplus_\infty Y$ enjoys the $\SSDTWOP{\kappa}$.
\end{corollary}
\begin{proof}
    Apply Theorem~\ref{theorem: ell_infty sums} with $|\mathscr A|=2$.
\end{proof}

\begin{example}\label{example: c_0(ell_n) SSD2P}
    Let $\kappa>\aleph_0$. We claim that $c_0(\mathbb N,\ell_n(\kappa))$ enjoys the $\SSDTWOP{\kappa}$, despite the fact that it is a sum of reflexive spaces. To this aim, fix $\varepsilon>0$ and choose any $m\in\mathbb N$ satisfying $2^\frac{1}{m}\le1+\varepsilon$.

    Now fix a set $A\subset S_{\ell_{m}(\kappa)^*}$ of cardinality $<\kappa$ and let $\{x^*_\alpha:\alpha\in\mathscr A\}$ be an enumeration for $A$. Moreover, find, for each $\alpha\in\mathscr A$, $x_\alpha\in S_{\ell_m(\kappa)}$ satisfying $x^*_\alpha(x_\alpha)\ge1-\varepsilon$.
    
    Since the support of the $x_\alpha$'s is at most countable and $\kappa>\aleph_0$, there exists an ordinal $\mu<\kappa$ such that $x_\alpha(\mu)=0$ holds for all $\alpha\in\mathscr A$. Call $y:=\delta_\lambda\in B_{\ell_m(\kappa)}$ and notice that
    \begin{equation*}
        \|x_\alpha\pm y\|=\left(\sum_{\lambda<\kappa}|x_\alpha(\lambda)|^m+1\right)^\frac{1}{m}=2^\frac{1}{m}\le1+\varepsilon
    \end{equation*}
    holds for every $\alpha\in\mathscr A$. Notice that, up to a small perturbation argument, we showed that the Banach spaces $\ell_n(\kappa)$'s satisfy the hypothesis of Theorem~\ref{theorem: c_0 sums}, thus the claim is proved.
    
    It is then clear that also the Banach space $\ell_\infty(\mathbb N,\ell_n(\kappa))$ enjoys the $\SSDTWOP{\kappa}$ thanks to Theorem~\ref{theorem: ell_infty sums}.
\end{example}

\begin{remark}\label{remark: c_0 sum not assd2p}
    One might wonder whether Theorem~\ref{theorem: c_0 sums} can be pushed further and used to obtain $c_0$ sums which possess the $\ASSDTWOP{\kappa}$. Unfortunately this doesn't happen, as a matter of fact, the space $c_0(\mathbb N,\ell_n(\kappa))$ fails the $\ASSDTWOP{\kappa}$ because, if by absurd it had the property, then Theorem~\ref{proposition: characterization k-ssd2p} would apply and this would lead to a contradiction when combined with Theorem~\ref{theorem: c_0 sums of uniformly convex spaces}.
\end{remark}
\subsection{Tensor product}\label{subsection: tensor product}

It is known that the $\ssdtwop$ is preserved by taking projective tensor products \cite[Theorem~2.2]{L2020_SymmetricStrongDiameterTwoPropertyInTensorProductsOfBanachSpaces}. In the cited paper, the authors' proof relies on the following characterization of the $\ssdtwop$.

\begin{fact}\label{fact: equivalent characterization SSD2P}\cite[Theorem~2.1]{HLLN2019_SymmetricStrongDiameterTwoProperty}
    Let $X$ be a Banach space. The following assertions are equivalent:
    \begin{itemize}
        \item[(i)] $X$ has the $\ssdtwop$.
        \item[(ii)] Given non-empty relatively weakly open sets $U_1,\ldots, U_n$ in $B_X$ and $\varepsilon>0$, there exist $x_1,\ldots, x_n, y\in B_X$ such that $\|y\|\ge1-\varepsilon$, $x_i\pm y\in B_X$ and $x_i\in U_i$ for all $1\le i\le n$.
        \item[(iii)] Given $x_1,\ldots x_n\in S_X$, there exist nets $(y^i_\alpha)$ and $(z_\alpha)$ in $S_X$ such that $\lim\|y^i_\alpha\pm z_\alpha\|=1$ and, with respect to the weak topology on $X$, $\lim z_\alpha=0$ and $\lim y^i_\alpha=x_i$ hold for all $1\le i\le n$.
    \end{itemize}
\end{fact}

As we will later demonstrate, Fact~\ref{fact: equivalent characterization SSD2P} doesn't hold true for the $\SSDTWOP{\kappa}$ whenever $\kappa>\aleph_0$. Therefore, a different proof is required to extend \cite[Theorem~2.2]{L2020_SymmetricStrongDiameterTwoPropertyInTensorProductsOfBanachSpaces} to the transfinite setting.

\begin{theorem}\label{theorem: projective tensor product}
    Let $X$ and $Y$ be Banach spaces and $\kappa>\aleph_0$. If $X$ and $Y$ have the $\SSDTWOP{\kappa}$, then the projective tensor product $X\hat{\otimes}_{\pi}Y$ enjoys the $\SSDTWOP{\kappa}$.
\end{theorem}
\begin{proof}
    Fix a set $\mathscr B\subset S_{(X\hat{\otimes}_{\pi}Y)^*}$ of cardinality $<\kappa$ and $\varepsilon>0$. Recall that the Banach space $(X\hat{\otimes}_{\pi}Y)^*$ is isometrically isomorphic to the space of bounded bilinear forms acting on $X\times Y$ \cite[Theorem~2.9]{R2002_Book_IntroductionToTensorProductsOfBanachSpaces}, hence, for every $B\in\mathscr B$, there exists $x_B\otimes y_B\in S_X\otimes S_Y$ satisfying $B(x_B\otimes y_B)\ge(1-\varepsilon)^\frac{1}{3}$.

    Given $B\in\mathscr B$, define
    \begin{equation*}
        B':=\frac{B(\cdot\otimes y_B)}{\|B(\cdot\otimes y_B)\|}\in S_{X^*}.
    \end{equation*}
    
    Since $X$ has the $\SSDTWOP{\kappa}$, there are $x$ and $x'_B$'s in $B_X$ such that $\|x\|\ge(1-\varepsilon)^\frac{1}{2}$ and, for all $B\in\mathscr B$, $x'_B\pm x\in B_X$ and $B'(x'_B)\ge(1-\varepsilon)^\frac{1}{3}$.
    
    Now, given $B\in\mathscr B$, call
    \begin{equation*}
        B'':=\frac{B(x'_B\otimes\cdot)}{\|B(x'_B\otimes\cdot)\|}\in S_{Y^*}.
    \end{equation*}

    Since $Y$ has the $\SSDTWOP{\kappa}$, there are $y$ and $y''_B$'s in $B_Y$ such that $\|y\|\ge(1-\varepsilon)^\frac{1}{2}$ and, for all $B\in\mathscr B$, $y''_B\pm y\in B_Y$ and $B''(y''_B)\ge(1-\varepsilon)^\frac{1}{3}$.

    Define $u_B:=x'_B\otimes y''_B\in B_{X\hat{\otimes}_{\pi}Y}$ and $v:=x\otimes y\in B_{X\hat{\otimes}_{\pi}Y}$. Notice that $\|v\|=\|x\|\|y\|\ge1-\varepsilon$, moreover, the fact that $u_B\pm v\in B_{X\hat{\otimes}_{\pi}Y}$ is due to \cite[Lemma~2.1]{L2020_SymmetricStrongDiameterTwoPropertyInTensorProductsOfBanachSpaces}. Eventually, let us prove that the set $\{u_B:B\in\mathscr B\}$ $(1-\varepsilon)$-norms $\mathscr B$.
    \begin{align*}
        B(u_B) & =B(x'_B\otimes y''_B)\ge(1-\varepsilon)^\frac{1}{3}\|B(x'_B\otimes\cdot)\|\ge(1-\varepsilon)^\frac{1}{3}B(x'_B\otimes y_B)\\
        & \ge(1-\varepsilon)^\frac{2}{3}\|B(\cdot\otimes y_B)\|\ge(1-\varepsilon)^\frac{2}{3}B(x_B\otimes y_B)\ge1-\varepsilon,
    \end{align*}
    which proves the claim and thus concludes the proof.
\end{proof}

\begin{remark}\label{remark: prohective tensor product}
    It is known that requiring only one component to have the $\ssdtwop$ is not enough in order to ensure the projective tensor product to enjoy the $\ssdtwop$ \cite[Corollary~3.9]{LLR2017_OctahedralNormsInTensorProductsOfBanachSpaces}. Up to a few changes, the same ideas can be used to show that requiring in the statement of Theorem~\ref{theorem: projective tensor product} only one component to enjoy the $\SSDTWOP{\kappa}$ is not enough. Let us sketch the argument required to prove this statement.

    We will later show that $\ell_\infty(\kappa)$ has the $\ASSDTWOP{\kappa}$ (see Example~\ref{example: C_0(X) spaces with SSD2P}), nevertheless, we claim that the Banach space $X:=\ell_\infty(\kappa)\hat\otimes_\pi\ell_3^3$ doesn't enjoy the $\SSDTWOP{\kappa}$.

    Since $\ell_3^3$ is not finitely representable in $\ell_1$, so it is not finitely representable in $\ell_1(\kappa)$ either (notice that each finite dimensional subspace of $\ell_1(\kappa)$ is isometrically isomorphic to some finite dimensional subspace of $\ell_1$, and vice-versa). Thanks to a simple transfinite analogue of \cite[Lemma~3.7]{LLR2017_OctahedralNormsInTensorProductsOfBanachSpaces} (replacing finite dimensional spaces with spaces of density $<\kappa$) we conclude that $\ell_1(\kappa)\hat\otimes_\varepsilon\ell_{3^*}^3$ is not $\kappa$-octahedral (see \cite[Definition~5.3]{CLL2022_Preprint_ACharacterizationOfBanachSpacesContainingEll1kappaViaBallCoveringProperties}), moreover, we can infer that $X=(\ell_1(\kappa)\hat\otimes_\varepsilon\ell_{3^*}^3)^*$ \cite[Theorem~5.3]{R2002_Book_IntroductionToTensorProductsOfBanachSpaces}. Therefore, by applying \cite[Theorem~3.2]{CLL2022_AttainingStrongDiameterTwoPropertyForInfiniteCardinals}, we conclude that $X$ fails the $\SSDTWOP{\kappa}$.
\end{remark}
\subsection{Some more remarks}\label{subsection: some more remarks}\hfill\\

Previously we claimed that the transfinite analogue of Fact~\ref{fact: equivalent characterization SSD2P} doesn't hold true. Let us now prove this statement for the implication (i)$\iff$(ii) by continuing the investigation began in Example~\ref{example: c_0(ell_n) SSD2P}. As a matter of fact, the Banach space $c_0(\mathbb N,\ell_n(\kappa))$ enjoys the $\SSDTWOP{\kappa}$, nevertheless we claim that it fails condition (ii) from Fact~\ref{fact: equivalent characterization SSD2P} with respect to $\aleph_1$. This claim follows from the following theorem.

\begin{theorem}\label{theorem: c_0 sums of uniformly convex spaces}
    Let $(X_n)$ be a sequence of Banach spaces. If, given any sequence of relatively weakly open sets $(U_n)$ in $B_{c_0(\mathbb N, X_n)}$ and $\varepsilon>0$, there exist $(x_n)$ and $y$ in $B_{c_0(\mathbb N, X_n)}$ such that $\|y\|\ge1-\varepsilon$, $x_n\pm y\in B_{c_0(\mathbb N, X_n)}$ and $x_n\in U_n$ for all $n\in\mathbb N$, then there exists $m\in\mathbb N$ such that $X_m$ is not uniformly convex.
\end{theorem}
\begin{proof}
    Let $A:=\{x^*_n:n\in\mathbb N\}\subset S_{\ell_1(\mathbb N, X_n^*)}$, where the $x^*_n$'s are any chosen elements satisfying the following conditions:
    \begin{equation*}
        x^*_n(m)\not=0\text{ and }\|x^*_n(n)\|\ge\|x^*_n(m)\|\text{ for all }n,m\in\mathbb N.
    \end{equation*}
    Now consider the relatively weakly open sets
    \begin{equation*}
        U_{n,m}:=\{x\in B_{c_0(\mathbb N, X_n)}:x^*_n(x)>1-m^{-1}\|x^*_n(m)\|\}
    \end{equation*}
    Fix $\varepsilon>0$ and find $x_{n,m}\in U_{n,m}$ and $y\in B_{c_0(\mathbb N, X_n)}$ such that $\|y\|\ge1-\varepsilon$ and $y\pm x_{n,m}\in B_{c_0(\mathbb N, X_n)}$ hold for all $n,m\in\mathbb N$.

    Since $\|y\|\ge1-\varepsilon$, we can find $p\in\mathbb N$ such that $\|y(p)\|\ge1-\varepsilon$. On the other hand, since $x_{p,m}\in U_{p,m}$, we have that
    \begin{align*}
        1-m^{-1}\|x^*_p(m)\| & \le x^*_p(x_{p,m})\le \sum_{n\not=p}\|x^*_p(n)\|+x^*_p(p)(x_{p,m}(p))\\
        & =1-\|x^*_p(p)\|+x^*_p(p)(x_{p,m}(p)),
    \end{align*}
    hence
    \begin{equation*}
        x^*_p(p)(x_{p,m}(p))\ge\|x^*_p(p)\|-m^{-1}\|x^*_p(m)\|,
    \end{equation*}
    therefore
    \begin{equation*}
        \|x_{p,m}(p)\|\ge1-m^{-1}\frac{\|x^*_p(m)\|}{\|x^*_p(p)\|}\ge1-m^{-1}.
    \end{equation*}
    Now, the fact that $x_{p,m}\pm y\in B_{c_0(\mathbb N, X_n)}$ implies
    \begin{equation*}
        1\ge\|x_{p,m}\pm y\|\ge\|x_{p,m}(p)\pm y(p)\|.
    \end{equation*}
    Finally, let us compute the modulus of convexity of $X_p$.
    \begin{align*}
        \delta_{X_p}(2-2\varepsilon) & :=\inf\left\{1-\left\|\frac{x+y}{2}\right\|:x,y\in B_{X_p}\text{ and }\|x-y\|\ge2-2\varepsilon\right\}\\
        & \le\inf_m\left(1-\frac{\|(x_{p,m}(p)+y(p))+(x_{p,m}(p)-y(p))\|}{2}\right)\\
        & =\inf_m(1-\|x_{p,m}(p)\|)\le\inf_m m^{-1}\\
        & =0,
    \end{align*}
    which implies that $X_p$ is not uniformly convex.
\end{proof}

Let us now turn our attention to the implication (i)$\iff$(iii) from Fact~\ref{fact: equivalent characterization SSD2P}. We claim that also this fails in the transfinite context.

\begin{example}\label{example: counterexample i iff iii}
    We will prove that $\ell_\infty(\kappa)$ fails the $\SSDTWOP{\kappa^+}$ (see Example~\ref{example: C_0(X) spaces with SSD2P}). Nevertheless, condition (iii) from Fact~\ref{fact: equivalent characterization SSD2P} is satisfied in a very strong way. In fact, fix $x\in\ S_{\ell_\infty(\kappa)}$, an ordinal $\mu<\kappa$ and call $y^x_\mu:=x-x(\mu)e_\mu\in B_{\ell_\infty(\kappa)}$ and $z_\mu=\delta_{\mu}\in S_{\ell_\infty(\kappa)}$. It is then clear that $y^x_\mu\pm z_\mu\in S_X$ and that, with respect to the weak topology, $\lim z_\mu=0$ and $\lim y^x_\mu=x$ holds for every $x\in S_{\ell_\infty(\kappa)}$. In other words, since $|\ell_\infty(\kappa)|=2^\kappa$, we showed that $\ell_\infty(\kappa)$ satisfies condition (iii) from Fact~\ref{fact: equivalent characterization SSD2P} with respect to $2^\kappa$.
\end{example}

Despite Theorem~\ref{theorem: c_0 sums of uniformly convex spaces} and Example~\ref{example: counterexample i iff iii}, it is possible to recover some transfinite analogue of Fact~\ref{fact: equivalent characterization SSD2P}, but only for the $\ASSDTWOP{\kappa}$.

\begin{proposition}\label{proposition: characterization k-ssd2p}
    Let $X$ be a Banach space and $\kappa>\aleph_0$. Consider the following statements:
    \begin{itemize}
        \item[(i)] $X$ has the $\ASSDTWOP{\kappa}$.
        \item[(ii)] Given a family $\mathscr U$ consisting of $<\kappa$ many relatively weakly open sets in $B_X$, a relatively weakly open neighborhood $V$ of $0$ in $B_X$ and $\varepsilon>0$, there are $\{x_U:U\in\mathscr U\}$ and $y\in V\cap S_X$ satisfying $x_U\in U$ and $x_u\pm y\in B_X$ for all $U\in\mathscr U$.
        \item[(iii)] Given $A\subset S_X$ of cardinality $<\kappa$, there are nets $\{(y^x_\alpha):x\in A\}$ and $(z_\alpha)$ in $S_X$ satisfying $\lim\|z_\alpha\pm y^x_\alpha\|=1$ and, with respect to the weak topology, $\lim z_\alpha=0$ and $\lim y^x_\alpha=x$ for all $x\in A$.
    \end{itemize}
    Then (i)$\implies$(ii)$\implies$(iii).
\end{proposition}
\begin{proof}
    (i)$\implies$(ii). Fix a family $\mathscr U$ consisting of $<\kappa$ many relatively weakly open sets in $B_X$, a relatively weakly open neighborhood $V$ of $0$ in $B_X$ and $\varepsilon>0$. For every $U\in\mathscr U$, thanks to Bourgain's lemma \cite[Lemma~II.1]{GGMS1987}, we can find functionals $x^*_{1,U},\ldots,x^*_{n_U,U}\in S_{X^*}$, $\varepsilon_U>0$ and convex coefficients $r_{1,U},\ldots,r_{n_U,U}$ such that
    \begin{equation*}
        \left\{\sum_{i=1}^{n_U}r_ix_i:x^*_{i,U}(x_i)>1-\varepsilon_U\text{ for all }1\le i\le n_U\right\}\subset U.
    \end{equation*}
    Moreover, we can find $x^*_{1,V},\ldots,x^*_{n_V,V}\in S_{X^*}$ and $\varepsilon_V>0$ satisfying
    \begin{equation*}
        \{x\in B_X:|x^*_{i,V}(x)|\le\varepsilon_V\text{ for all }1\le i\le n_V\}\subset V.
    \end{equation*}
    Since $X$ has the $\ASSDTWOP{\kappa}$ and $|\{x^*_{i,U}:1\le i\le n_U\text{ and }U\in\mathscr U\cup\{V\}\}|\le\aleph_0\cdot|\mathscr U|<\kappa$, there exist $\{x_{i,U}:1\le i\le n_U\text{ and }U\in\mathscr U\cup\{V\}\}$ and $y$ in $S_X$ satisfying $x^*_{i,U}(x_{i,U})\ge1-\varepsilon_U$ and $x_{i,U}\pm y\in S_X$ for all $1\le i\le n_U$ and $U\in\mathscr U\cup\{V\}$. Now, given $U\in\mathscr U$, define
    \begin{equation*}
        x_U:=\sum_{i=1}^{n_U}r_ix_{i,U}\in B_X
    \end{equation*}
    and notice that $x_U\in U$. Moreover,
    \begin{equation*}
        \|x_U\pm y\|=\left\|\sum_{i=1}^{n_U}r_i(x_{i,U}\pm y)\right\|\le\sum_{i=1}^{n_U}r_i\|x_{i,U}\pm y\|=1.
    \end{equation*}

    In order to conclude, it only remains to prove that $y\in V$. But this is clear because, for every $1\le i\le n_V$ we have that
    \begin{equation*}
        1=\|x_{i,V}\pm y\|\ge x^*_{i,V}(x_{i,V}\pm y)\ge1-\varepsilon_V\pm x^*_{i,V}(y),
    \end{equation*}
    which means that $|x^*_{i,V}(y)|\le\varepsilon_V$, hence $y\in V$.

    (ii)$\implies$(iii). Fix a set $A\subset S_X$ of cardinality $<\kappa$ and temporarily fix a weak neighborhood $U$ of $0$. Define $\mathscr U:=\{(x+U)\cap B_X:x\in A\}$ and find $\{y_U^x:x\in A\}\subset B_X$ and $z_U\in U\cap S_X$ satisfying $y_U^x\in x+U$ and $y_U^x\pm z_U\in B_X$ for all $x\in A$.

    Now semi-order the family of weakly open neighborhoods of $0$ with respect to the inclusion and consider the nets $(y_U^x)$ and $(z_U)$. It is clear that $\lim y^x_U=x$, $\lim z_U=0$ and $\lim\|z_U\pm y^x_U\|=1$ holds for all $x\in A$. Moreover, up to a perturbation argument, we can assume that all $y_U^x$'s belong to $S_X$. Thus the claim is proved.
\end{proof}

Let us show that the implication (iii)$\implies$(ii) from Proposition~\ref{proposition: characterization k-ssd2p} fails. As already witnessed by Example~\ref{example: counterexample i iff iii}, $\ell_\infty(\kappa)$ satisfies condition (iii) in a very strong way, nevertheless it fails the $\SSDTWOP{\kappa^+}$. Therefore, we only need to notice that condition (ii) with respect to $\kappa^+$ clearly implies possessing the $\SSDTWOP{\kappa^+}$, thus the claim is proved.

It remains unclear whether the implication (ii)$\implies$(i) hold.
\section{$C_0(X)$ spaces}\label{section: C_0(X) spaces}

In \cite{ANP2019_NewApplicationsOfExtremelyRegularFunctionSpaces} it was proved that $C_0(X)$, for $X$ infinite Hausdorff locally compact, always has the $\ssdtwop$. In this section we aim to extend the class of examples which enjoy the transfinite $\ssdtwop$ by trying to characterize under which conditions $C_0(X)$ spaces have this property. Before doing so, let us introduce a bit of notation about some cardinal functions.

Let $X$ be a topological space. Define the \emph{density character} of $X$ as
\begin{equation*}
    \density(X):=\min\{|\mathscr D|:\mathscr D\subset X\text{ is dense}\}+\aleph_0.
\end{equation*}
A cellular family in $X$ is a family of mutually disjoint open sets in $X$. Define the \emph{cellularity} of $X$ as
\begin{equation*}
    \cellularity(X):=\sup\{|\mathscr C|:\mathscr C\text{ is a cellular family in }X\}+\aleph_0.
\end{equation*}
It is well known that $\cellularity(X)\le\density(X)$. We refer the reader to \cite{KV1984_Book_HandbookOfSetTheoreticTopology} for a detailed treatment about these cardinal functions and more.

Before stating the main result of this section, let us recall that, thanks to the Riesz–Markov representation theorem, every continuous linear functional on $C_0(X)$ admits a unique representation as a regular countably additive Borel measure on $X$.

\begin{theorem}\label{theorem: C_0(X) SSD2P}
    Let $X$ be a $T_4$ locally compact space.
    \begin{itemize}
        \item[(i)] $C_0(X)$ fails the $\SSDTWOP{\density(X)^+}$.
        \item[(ii)] If $\cellularity(X)>\aleph_0$, then $C_0(X)$ has the $\ASSDTWOP{\cellularity(X)}$.
    \end{itemize}
\end{theorem}
\begin{proof}
    $(i)$. Let $\mathscr D$ be dense in $X$. Consider the set $\{\delta_x:x\in\mathscr D\}\subset S_{C_0(X)^*}$ and suppose by contradiction that $C_0(X)$ has the $\SSDTWOP{\density(X)^+}$. Then we can find functions $\{f_x:x\in\mathscr D\}\subset B_{C_0(X)}$ and $g\in B_{C_0(X)}$ satisfying
    \begin{equation*}
        \|g\|\ge2/3,\ f_x(x)\ge2/3\text{ and }\|f_x\pm g\|\le 1.
    \end{equation*}
    Since $\mathscr D$ is dense, then we can find $x\in\mathscr D$ such that $|g(x)|>1/3$, which contradicts the fact that $|f_x(x)\pm g(x)|\le1$.

    $(ii)$. Fix $\lambda<\cellularity(X)$ and a set $\mathscr M\subset S_{C_0(X)^*}$ of cardinality $\lambda$. Find a cellular family $\mathscr C$ in $X$ of size $\lambda<|\mathscr C|\le\cellularity(X)$ and, given any $m\in\mathbb N$ and $\mu\in\mathscr M$, define
    \begin{equation*}
        \mathscr C_{m,\mu}:=\{C\in\mathscr C:|\mu|(C)>m^{-1}\}.
    \end{equation*}
    Notice that
    \begin{equation*}
        |\{\mathscr C_{m,\mu}:m\in\mathbb N\text{ and }\mu\in\mathscr M\}|\le\aleph_0\cdot\lambda<|\mathscr C|.
    \end{equation*}
    Therefore, there is $C\in\mathscr C$ satisfying $|\mu|(C)=0$ for every $\mu\in\mathscr M$. Notice that, without loss of generality, we can assume that $|\mu|(\overline{C})=0$. In fact, if that's not the case, then we can replace $C$ with some non-empty open set $C'$ satisfying $\overline{C'}\subset C$.

    Find functions $\{f_{m,\mu}:m\in\mathbb N\text{ and }\mu\in\mathscr M\}\subset S_{C_0(X)}$ such that $\mu(f_{m,\mu})\ge1-(3m)^{-1}$ and, since $\mu$'s are regular, compact sets $\{K_{m,\mu}:m\in\mathbb N\text{ and }\mu\in\mathscr M\}\subset X\setminus\overline{C}$ satisfying $|\mu|(K_{m,\mu})\ge1-(3m)^{-1}$. Now construct Urysohn's functions $\{g_{m,\mu}:m\in\mathbb N\text{ and }\mu\in\mathscr M\}$ and $h$ in $S_{C_0(X)}$ satisfying
    \begin{equation*}
        g_{m,\mu}|_{K_{m,\mu}}=1,\ g_{m,\mu}|_{\overline{C}}=0\text{ and }h|_{X\setminus C}=0.
    \end{equation*}
    Define 
    \begin{equation*}
        i_{m,\mu}:=\frac{f_{m,\mu}\cdot g_{m,\mu}}{\|f_{m,\mu}\cdot g_{m,\mu}\|}\in S_{C_0(X)}
    \end{equation*}
    and notice that $i_{m,\mu}\pm h\in S_{C_0(X)}$. Moreover, given any $m\in\mathbb N$ and $\mu\in\mathscr M$,
    \begin{align*}
        \mu(i_{m,\mu}) & \ge\int_X f_{m,\mu}\cdot g_{m,\mu}d\mu\ge\int_{K_{m,\mu}}f_{m,\mu}d\mu-(3m)^{-1}\\
        & \ge\int_X f_{m,\mu}d\mu-2\cdot(3m)^{-1}\ge1-m^{-1}.
    \end{align*}
\end{proof}

It remains unclear whether the statement of Theorem~\ref{theorem: C_0(X) SSD2P} can be written using only one cardinal function. Namely, we don't know the answer to the following two questions.

\begin{question}
    Let $X$ be a $T_4$ locally compact space. Is it true that $C_0(X)$ fails the $\SSDTWOP{\cellularity(X)^+}$? Is it true that $C_0(X)$ enjoys the $\ASSDTWOP{\density(X)}$, whenever $\density(X)>\aleph_0$?
\end{question}

\begin{example}\label{example: C_0(X) spaces with SSD2P}
    Let us now employ Theorem~\ref{theorem: C_0(X) SSD2P} to produce some new examples of spaces enjoying or failing the transfinite $\ssdtwop$.
    \begin{itemize}
        \item[(i)] Let $X$ be a separable locally compact Hausdorff space. It is clear that $\cellularity(X)\le\density(X)=\aleph_0$, hence $C_0(X)$ fails the $\SSDTWOP{\aleph_1}$.
        \item[(ii)] It is known that $\cellularity(\beta\mathbb N\setminus\mathbb N)=2^{\aleph_0}$ \cite[7.22]{KV1984_Book_HandbookOfSetTheoreticTopology}, therefore $C(\beta\mathbb N\setminus\mathbb N)$ enjoys the $\ASSDTWOP{2^{\aleph_0}}$.
        \item[(iii)] Let $B$ be a Boolean algebra and let $S(B)$ be the Stone space associated to $B$. It is clear that the set $\{\{b\}:b\in B\}\subset S(B)$ defines a cellular family in $S(B)$.
        
        Now let us consider a regular positive Borel measure $\mu$ over some T$_4$ locally compact space $X$. Call $\mathfrak B_\mu$ the set of measurable sets modulo the negligible sets in $X$. It is known that $L_\infty(\mu)$ is isometrically isomorphic to $C(S(\mathfrak B_\mu))$ (see e.g. pages~27--29 in \cite{DLS2012}), therefore we conclude that $L_\infty(\mu)$ enjoys the $\ASSDTWOP{|\mathfrak B_\mu|}$, whenever $|\mathfrak B_\mu|>\aleph_0$.

        In particular, whenever $\kappa>\aleph_0$ and $\mu$ is the counting measure over $\kappa$, $|\mathfrak B_\mu|=\kappa$, thus it follows that $\ell_\infty(\kappa)$ enjoys the $\ASSDTWOP{\kappa}$, but it fails the $\SSDTWOP{\kappa^+}$, because $\density(\ell_\infty(\kappa))=\kappa$.
    \end{itemize}
\end{example}

To conclude this section, let us provide a criterion to identify cellular families in particular classes of topological spaces, including Alexandrov-discrete spaces.

\begin{proposition}\label{proposition: criterion cellular families}
    Let $X$ be a $T_{2\frac{1}{2}}$ space and $\kappa$ an infinite cardinal. If there are $\kappa$ many points in $X$ such that every non-empty intersection of at most $\kappa$ many neighborhoods is still a neighborhood, then $\cellularity(X)\ge\kappa$.
\end{proposition}
\begin{proof}
    Let $ A\subset X$ be a set of cardinality $\kappa$ such that every non-empty intersection of at most $\kappa$ many neighborhoods of $x$ is still a neighborhood for every $x\in A$. For every distinct $x,y\in A$ find a closed neighborhood $U_{x,y}$ of $x$ which doesn't contain $y$. By assumption
    \begin{equation*}
        U_x:=\left(\bigcap_{y\in A\setminus\{x\}}U_{x,y}^\circ\right)\cap\left(\bigcap_{y\in A\setminus\{x\}}X\setminus U_{y,x}\right)
    \end{equation*}
    is an open neighborhood of $x$. Notice that, given distinct $x,y\in A$ we have that
    \begin{equation*}
        U_x\cap U_y\subset U_{x,y}^\circ\cap(X\setminus U_{y,x})\cap U_{y,x}^\circ\cap(X\setminus U_{x,y})=\emptyset
    \end{equation*}
    In other words, $\{U_x:x\in A\}$ defines a cellular family of size $\kappa$.
\end{proof}

Notice that the assumption in Proposition~\ref{proposition: criterion cellular families} is far from being necessary. As a matter of fact, it is consistent with ZFC that $\beta\mathbb N\setminus\mathbb N$ contains no P-points, that is, points for which every $G_\delta$ containing them is a neighborhood, nevertheless, as already recalled in Example~\ref{example: C_0(X) spaces with SSD2P}, $\beta\mathbb N\setminus\mathbb N$ has a cellular family of cardinality $2^{\aleph_0}$.

\end{document}